\documentclass[a4paper,12pt]{article}
\usepackage[T1]{fontenc}
\usepackage[utf8]{inputenc}
\usepackage{lmodern}
\usepackage{amssymb}
\usepackage{amsthm}
\usepackage{amsmath}
\usepackage{color}
\usepackage{mathrsfs}

\theoremstyle{definition}
\newtheorem{definition}{Definition}

\theoremstyle{definition}
\newtheorem{example}{Example}

\theoremstyle{definition}
\newtheorem{remark}{Remark}

\theoremstyle{plain}
\newtheorem{theorem}{Theorem}

\theoremstyle{plain}
\newtheorem{corollary}{Corollary}

\theoremstyle{plain}
\newtheorem{lemma}{Lemma}

\newcommand{\Span}[1]{\langle #1 \rangle}
\newcommand{\vf}[1]{\mathcal{D}(#1)}

\title{Projectable Lie algebras of vector fields in 3D}
\author{Eivind Schneider}

\begin{document}

\maketitle

\begin{abstract}
Starting with Lie's classification of finite-dimensional transitive Lie algebras of vector fields on $\mathbb C^2$ we construct Lie algebras of vector fields on the bundle $\mathbb C^2 \times \mathbb C$ by lifting the Lie algebras from the base. There are essentially three types of transitive lifts and we compute all of them for the Lie algebras from Lie's classification. The simplest type of lift is encoded by Lie algebra cohomology. 
\end{abstract}

\section{Introduction}
A fundamental question in differential geometry is to determine which transitive Lie group actions exist on a manifold. Sophus Lie considered this to be an important problem, in particular due to its applications in the symmetry theory of PDEs. 
In \cite{Lie1880} (see also \cite{Transformationsgruppen3}) he gave a local classification of finite-dimensional transitive Lie algebras of analytic vector fields on $\mathbb C$ and $\mathbb C^2$.

Lie never published a complete list of finite-dimensional Lie algebras of vector fields on $\mathbb C^3$, but he did classify primitive Lie algebras of vector fields on $\mathbb C^3$, those without an invariant foliation, which he considered to be the most important ones and also some special imprimitive Lie algebras of vector fields.

Lie algebras of vector fields on $\mathbb C^3$ preserving a one-dimensional foliation are locally equivalent to projectable Lie algebras of vector fields on the total space of the fiber bundle $\pi \colon \mathbb C^2 \times \mathbb C \to \mathbb C^2$. Finding such Lie algebras amounts to extending Lie algebras of vector fields on the base (where they have been classified) to the total space. For the primitive Lie algebras of vector fields on the plane, this was completed by Lie \cite{Transformationsgruppen3}. Amaldi continued Lie's work by extending the imprimitive Lie algebras to three-dimensional space \cite{Amaldi1, Amaldi2} (see also \cite{Hillgarter1}). 
Nonsolvable Lie algebras of vector fields on $\mathbb C^3$ were classified in \cite{Doubrov}. It was also showed there that a complete classification of finite-dimensional solvable Lie algebras of vector fields on $\mathbb C^3$ is hopeless, since it contains the subproblem of classifying left ideals of finite codimension in the universal enveloping algebra $U(\mathfrak g)$ for the two-dimensional Lie algebras $\mathfrak g$ which is known to be a hard algebraic problem.

In this paper we consider Lie algebras of vector fields on the plane from Lie's list, and extend them to the total space $\mathbb C^2 \times \mathbb C$. In order to avoid the issues discussed in \cite{Doubrov} we only consider extensions that are of the same dimension as the original Lie algebra. The resulting list of Lie algebras has intersections with \cite{Transformationsgruppen3}, \cite{Amaldi1, Amaldi2} and \cite{Doubrov}, but it also contains some new solvable Lie algebras of vector fields in three-dimensional space. 

We start in section \ref{classification2D} by reviewing the classification of Lie algebras of vector fields on $\mathbb C^2$, which will be our starting point. The lifting procedure is explained in section \ref{lifts}. We show that the lifts can be divided into three types, depending on how they act on the fibers of $\pi$. In section \ref{list} we give a complete list of the lifted Lie algebras of vector fields, which is the main result of this paper. The relation between the simplest type of lift and Lie algebra cohomology is explained in section \ref{cohomology}. 

\section{Classification of Lie algebra actions on $\mathbb C^2$}\label{classification2D}
Two Lie algebras $\mathfrak g_1 \subset \vf{M_1}, \mathfrak g_2 \subset \vf{M_2}$ of vector fields are locally equivalent if there exist open subsets $U_i \subset M_i$ and a diffeomorphism $f \colon U_1 \to U_2$ with the property $df(\mathfrak g_1|_{U_1})=\mathfrak g_2|_{U_1}$. Recall that $\mathfrak g$ is transitive if $\mathfrak g|_p =T_p M$ at all points $p \in M$.

The classification of Lie algebras of vector fields on $\mathbb C$ and $\mathbb C^2$ is due to Lie \cite{Lie1880} (see \cite{AH} for English translation). There are up to local equivalence only three finite-dimensional transitive Lie algebras of vector fields on $\mathbb C$ and they correspond to the metric, affine and projective transformations: 
\begin{equation}
\Span{\partial_u}, \qquad \Span{\partial_u, u \partial_u}, \qquad \Span{\partial_u, u \partial_u, u^2 \partial_u} \label{classification1D}
\end{equation} 
On $\mathbb C^2$ any finite-dimensional transitive Lie algebra of analytic vector fields is locally equivalent to one of the following:
	\begin{align*}
	&\textbf{Primitive}\\ % and locally transitive
	\mathfrak{g}_1 &= \Span{\partial_x, \partial_y, x \partial_x, x \partial_y, y \partial_x, y \partial_y, x^2 \partial_x +xy \partial_y, xy \partial_x +y^2 \partial_y}\\
	\mathfrak{g}_2 &= \Span{\partial_x,\partial_y, x\partial_x, x \partial_y, y\partial_x, y \partial_y} \\
	\mathfrak{g}_3 &= \Span{\partial_x, \partial_y, x \partial_y, y \partial_x, x \partial_x - y\partial_y}\\
	\end{align*}
	\begin{align*}
	&\textbf{Imprimitive}\\ % , locally transitive
	\mathfrak{g}_4 &=\Span{\partial_x, e^{\alpha_i x} \partial_y, xe^{\alpha_i x} \partial_y, ..., x^{m_i-1}e^{\alpha_i x} \partial_y\mid i=1,...,s},\\ &\qquad \text{where } m_i \in \mathbb N\setminus\{0\}, \alpha_i \in \mathbb C,  \sum_{i=1}^s m_i + 1 = r \geq 2\\ 
	\mathfrak{g}_5 &=\Span{\partial_x,y \partial_y,e^{\alpha_i x} \partial_y, xe^{\alpha_i x} \partial_y, ..., x^{m_i-1}e^{\alpha_i x} \partial_y\mid i=1,...,s},\\ &\qquad \text{where } m_i \in \mathbb N\setminus\{0\}, \alpha_i \in \mathbb C,  \sum_{i=1}^s m_i + 2 = r \geq 4 \\
	%\mathfrak{g}_4 &=\Span{\partial_x, e^{\alpha_1 x} \partial_y, xe^{\alpha_1 x} \partial_y, ..., x^{m_1-1}e^{\alpha_1 x} \partial_y,e^{\alpha_2 x} \partial_y,..., x^{m_s-1}e^{\alpha_s x} \partial_y },\\ &\qquad \text{where } m_i \in \mathbb N, \alpha_i \in \mathbb C, i=1,...,s, \sum_{i=1}^s m_i + 1 = r \geq 3\\ 
	%\mathfrak{g}_5 &=\Span{\partial_x,y \partial_y, e^{\alpha_1 x} \partial_y, xe^{\alpha_1 x} \partial_y, ..., x^{m_1-1}e^{\alpha_1 x} \partial_y,xe^{\alpha_2 x} \partial_y,..., x^{m_s-1}e^{\alpha_s x} \partial_y },\\ &\qquad \text{where } m_i \in \mathbb N, \alpha_i \in \mathbb C, i=1,...,s, \sum_{i=1}^s m_i + 2 = r \geq 4 \\%\geq 2\\ 
	\mathfrak{g}_6 &= \Span{\partial_x, \partial_y, y \partial_y, y^2 \partial_y}  \\
	\mathfrak{g}_7 &= \Span{\partial_x, \partial_y, x\partial_x, x^2 \partial_x + x \partial_y}\\
	\mathfrak{g}_8 &=  \Span{\partial_x, \partial_y, x \partial_y, ..., x^{r-3} \partial_y, x \partial_x+ \alpha y \partial_y\mid \alpha \in \mathbb C} ,\; r \geq 3\\
	\mathfrak{g}_9 &= \Span{\partial_x, \partial_y, x \partial_y, ..., x^{r-3} \partial_y, x \partial_x + \left( (r-2) y+ x^{r-2}\right) \partial_y },\; r \geq 3\\ 
	\mathfrak{g}_{10} &= \Span{\partial_x, \partial_y, x \partial_y, ..., x^{r-4} \partial_y, x \partial_x, y \partial_y},\; r \geq 4\\
	\mathfrak{g}_{11} &= \Span{\partial_x, x \partial_x, \partial_y, y\partial_y, y^2 \partial_y}\\
		\mathfrak{g}_{12} &= \Span{\partial_x, x \partial_x, x^2 \partial_x, \partial_y, y \partial_y, y^2 \partial_y }\\
	%{\color{red}\mathfrak{g}_{12} }&= \Span{\partial_x, x \partial_x + \partial_y }\\
	\mathfrak{g}_{13} &= \Span{\partial_x, \partial_y, x \partial_y, ..., x^{r-4} \partial_y, x^2 \partial_x + (r-4) xy \partial_y, x \partial_x + \tfrac{r-4}{2} y \partial_y},\; r \geq 5\\
	\mathfrak{g}_{14} &= \Span{ \partial_x, \partial_y, x \partial_y,..., x^{r-5} \partial_y, y \partial_y, x \partial_x, x^2 \partial_x + (r-5)xy \partial_y}, \;r\geq 6\\ 
	\mathfrak{g}_{15} &= \Span{\partial_x, x\partial_x+\partial_y, x^2 \partial_x + 2x \partial_y }\\ 
	\mathfrak{g}_{16} &= \Span{\partial_x, x \partial_x- y \partial_y, x^2 \partial_x+(1-2xy) \partial_y} \\ 
	\end{align*}

In the list above (which is based on the one in \cite{Onishchik}), and throughout the paper, $r$ denotes the dimension of the Lie algebra. Our $\mathfrak g_{16}$ is by $y\mapsto \frac{1}{y-x}$ locally equivalent to $\Span{\partial_x + \partial_y, x \partial_x + y \partial_y,x^2 \partial_x + y^2 \partial_y}$, which often appears in these lists of Lie algebras of vector fields on the plane but has a singular orbit $y-x=0$. 

We also refer to \cite{Transformationsgruppen3,Campbell,Draisma,Gorbatsevich} which treat transitive Lie algebras of vector fields on the plane.

\section{Lifts of Lie algebras in $\vf{\mathbb C^2}$}\label{lifts}
In this section we describe how we lift the Lie algebras of vector fields from the base to the total space of $\pi \colon \mathbb C^2 \times \mathbb C \to \mathbb C^2$.

\begin{definition}
	Let $\mathfrak g \subset \vf{\mathbb C^2}$ be a Lie algebra of vector fields on $\mathbb C^2$, and let $\hat{\mathfrak{g}} \subset \vf{\mathbb C^2 \times \mathbb C}$ be a projectable Lie algebra satisfying $d \pi(\hat{\mathfrak{g}})= \mathfrak g$. 
	The Lie algebra $\hat{\mathfrak{g}}$ is a lift of $\mathfrak g$ (on the bundle $\pi$) if $\ker (d\pi|_{\hat{\mathfrak{g}}})=\{0\}$ .  
\end{definition}

For practical purposes we reformulate this in coordinates. Throughout the paper $(x,y,u)$ will be coordinates on $\mathbb C^2\times \mathbb C$. If $X_i=a_i(x,y) \partial_x+b_i(x,y) \partial_y$ form a basis for $\mathfrak g \subset \vf{\mathbb C^2}$, then a lift $\hat{\mathfrak g}$ of $\mathfrak g$ on the bundle $\pi$ is spanned by vector fields of the form $\hat X_i=a_i(x,y) \partial_x+b_i(x,y) \partial_y+f_i(x,y,u) \partial_u$. The functions $f_i$ are subject to differential constraints coming from the commutation relations of $\mathfrak g$.
Finding lifts of $\mathfrak g$ amounts to solving these differential equations.

\subsection{Three types of lifts}
The fibers of $\pi$ are one-dimensional and, as is common in these type of calculations, we will use the classification of Lie algebras of vector fields on the line to simplify our calculations. Let $\mathfrak g$ be a finite-dimensional transitive Lie algebra of vector fields on $\mathbb C^2$ and $\hat{\mathfrak g}$ a transitive lift. For $p \in \mathbb C^2 \times \mathbb C$, let $a=\pi(p)$ be the projection of $p$ and let $\mathfrak{st}_a \subset \mathfrak g$ be the stabilizer of $a\in \mathbb C^2$. Denote by $\hat{\mathfrak{st}_a} \subset \hat{\mathfrak{g}}$  the lift of $\mathfrak{st}_a$, i.e.  $d\pi(\hat{\mathfrak{st}_a})=\mathfrak{st}_a$. The Lie algebra $\hat{\mathfrak{st}_a}$ preserves the fiber $F_a$ over $a$, and thus induces a Lie algebra of vector fields on $F_a$ by restriction to the fiber. Denote the corresponding Lie algebra homomorphism by \[\varphi_a \colon \hat{\mathfrak{st}_a}\to \vf{F_a}.\]
In general this will not be injective, and it is clear that as abstract Lie algebras $\varphi_a(\hat{\mathfrak{st}_a})$ is isomorphic to $\mathfrak h_a= \hat{\mathfrak{st}_a}/\ker(\varphi_a)$.

Since $\hat{\mathfrak{g}}$ is transitive, the Lie algebra $\varphi_a(\hat{\mathfrak{st}_a})$ is a transitive Lie algebra on the one-dimensional fiber $F_a$, and therefore it must be locally equivalent to one of the three Lie algebras (\ref{classification1D}).
Transitivity of $\hat{\mathfrak g}$ also implies that for any two points $a,b \in \mathbb C^2$, the Lie algebras $\varphi_a(\hat{\mathfrak{st}_{a}}), \varphi_b(\hat{\mathfrak{st}_{b}})$ of vector fields are locally equivalent. Since the Lie algebra structure of $\mathfrak h_a$ is independent of the point $a$, it will be convenient to define $\mathfrak h$ as the abstract Lie algebra isomorphic to $\mathfrak h_a$. Thus $\dim \mathfrak h$ is equal to 1, 2 or 3, which allows us to split the transitive lifts into three distinct types. 

The main goal of this section is to show that we can change coordinates so that $\varphi_a(\hat{\mathfrak{st}_a})$ has one of the three normal forms from (\ref{classification1D}), on every fiber simultaneously. Before we prove this we make the following observation.

\begin{lemma} \label{lemma}
	Let $\mathfrak g \subset \vf{\mathbb C^2}$ be a transitive Lie algebra of vector fields, and let $a \in \mathbb C^2$ be an arbitrary point. Then there exists a locally transitive two-dimensional subalgebra  $\mathfrak h \subset \mathfrak{g}$, and a local coordinate chart $(U,(x,y))$ centered at $a$ such that $\mathfrak h = \Span{X_1,X_2}$ where  $X_1= \partial_x$ and either $X_2=\partial_y$ or $X_2= x \partial_x+\partial_y$. 
\end{lemma}
\begin{proof}
	This is apparent from the list in section \ref{classification2D}, but we also outline an independent argument.
	It is well known that a two-dimensional locally transitive Lie subalgebra can be brought to one of the above forms, so we only need to show that such exists.
	
	Let $\mathfrak g = \mathfrak s \ltimes \mathfrak r$ be the Levi-decomposition of $\mathfrak g$. Assume first that $\mathfrak r$ is a locally transitive Lie subalgebra and let \[ \mathfrak r \supset \mathfrak r_1 \supset \mathfrak r_2 \supset \cdots \supset \mathfrak r_k  \supset \mathfrak r_{k+1}=\{0\}.   \] 
	be its derived series. If $\mathfrak r_k$ is locally transitive, it contains an (abelian) two-dimensional transitive subalgebra and we are done. If $\mathfrak r_k$ is not locally transitive, then we take a vector field $X_i \in \mathfrak r_i$ for some $i<k$ which is transversal to those of $\mathfrak r_k$. Since $[\mathfrak r_i, \mathfrak r_k] \subset \mathfrak r_k$, we have a map $\text{ad}_{X_i} \colon \mathfrak r_k \to \mathfrak r_k$. Let $X_k \in \mathfrak r_k$ be an eigenvector of $\text{ad}_{X_i}$. Then $X_i$ and $X_k$ span a two-dimensional locally transitive subalgebra of $\mathfrak g$. 

	If $\mathfrak s$ is a transitive subalgebra, then $\mathfrak s$ is locally equivalent to the standard realization on $\mathbb C^2$ of either $sl_2$, $sl_2 \oplus sl_2$ or $sl_3$, all of which have a locally transitive two-dimensional Lie subalgebra. 
	
	If neither $\mathfrak s$ nor $\mathfrak r$ is locally transitive they give transversal one-dimensional foliations, and $\mathfrak s$ is locally equivalent to the realization $\Span{\partial_x, x \partial_x, x^2 \partial_x}$ of  $sl_2$ on $\mathbb C$ while $\mathfrak r$ is spanned by vector fields of the form $b_i(x,y)\partial_y$. Since $\mathfrak r$ is finite-dimensional we get $(b_i)_x=0$. Therefore $\mathfrak g=\mathfrak s\oplus \mathfrak r$, and there exists an abelian locally transitive subalgebra.  
\end{proof}

\begin{example} \label{simplify}
	Let $X_1= \partial_x$ and $X_2= \partial_y$ be  vector fields on $\mathbb C^2$ and consider the general lift $\hat{X}_1 = \partial_x+f_1(x,y,u) \partial_u, \hat X_2 = \partial_y+f_2(x,y,u) \partial_u$. We may change coordinates $ u\mapsto A(x,y,u)$ so that $f_1\equiv 0$. This amounts to solving $\hat X_1 (A)=A_x+f_1 A_u=0$ with $A_u \neq 0$, which can be done locally around any point.
	The commutation relation $[\hat X_1, \hat X_2]=(f_2)_x \partial_u=0$ implies that $f_2$ is independent of $x$. Thus, in the same way as above, we may change coordinates $u \mapsto B(y,u)$ so that $f_2 \equiv 0$. A similar argument works if $X_2= x \partial_x+\partial_y$. 
\end{example}
The previous example is both simple and useful. Since all our Lie algebras of vector fields on $\mathbb C^2$ contain these Lie algebras as subalgebras, we can always transform our lifts to a simpler form by changing coordinates in this way. We apply this idea in the proof of the following theorem.

\begin{theorem}\label{main}
Let $\mathfrak g = \Span{X_1,...,X_r}$ be a transitive Lie algebra of vector fields on $\mathbb C^2$ and let $\hat{\mathfrak g}=\Span{\hat{X_1},...,\hat{X_r}}$ be a transitive lift of $\mathfrak g$ on the bundle $\pi$, with $\hat{X_i} = X_i + f_i(x,y,u) \partial_u$. 
	
	Then there exist local coordinates in a neighborhood $U \subset \mathbb C^2 \times \mathbb C$ of any point so that $f_i(x,y,u)=\alpha_i(x,y) + \beta_i(x,y) u+\gamma_i (x,y) u^2$ and $\varphi_a(\hat{\mathfrak{st}_{a}})$ is of normal form (\ref{classification1D}) for every $a\in U$.
\end{theorem}
\begin{proof}
		Let $p \in \mathbb C^2 \times \mathbb C$ be an arbitrary point, $V$ an open set containing $p$, and $(V,(x,y,u))$ a coordinate chart centered at $p$. By lemma \ref{lemma} we may assume that $X_1=\partial_x$ and either $X_2=\partial_y$ or $X_2=x\partial_x+\partial_y$ and by example \ref{simplify} we may set $f_1 \equiv 0 \equiv f_2$.   We choose a basis of $\mathfrak g$ such that $\mathfrak{st}_0=\Span{X_3,...,X_r}$. 
		
	Since $\varphi_0(\hat{\mathfrak{st}_0})$ is a transitive action on the line, we may in addition make a local coordinate change $u \mapsto A(u)$ on $U \subset V$ containing $0$ so that $\varphi_0(\hat{\mathfrak{st}_0})$ is of the form $\Span{\partial_u}$, $\Span{\partial_u, u \partial_u}$ or $\Span{\partial_u, u \partial_u, u^2 \partial_u}$. Then for $i=3,...,r$, the functions $f_i$ have the property \[f_i(0,0,u)= \tilde \alpha_i+\tilde \beta_i u+\tilde \gamma_i u^2.\]  We use the commutation relations of $\hat{\mathfrak g}$ to show that $f_i(x,y,u)$ will take this form for any $(x,y,u) \in U$. 
	
	If $[X_j,X_i]=c_{ji}^k X_k$ are the commutation relations for $\mathfrak g$, then the lift of $\mathfrak g$ obeys the same relations: $[\hat X_j, \hat X_i] = c_{ji}^k \hat X^k$. Thus 
	\[ [\hat X_1, \hat X_i] = [X_1,X_i]+X_1(f_i) \partial_u= c_{1i}^k X_k+X_1(f_i)\partial_u \] which implies that $X_1(f_i)=c_{1i}^k f_k$. In the same manner we get the equations $X_2(f_i)=c_{2i}^k f_k$. We can rewrite the equations as
	\[\partial_x(f_i)=c_{1i}^k f_k, \qquad \partial_y(f_i)=\tilde c_{2i}^k(x) f_k.   \]
	The coefficients $\tilde c_{2i}^k(x)$ depend on whether $\Span{X_1,X_2}$ is abelian or not, but in any case they are indepedent of $u$. We differentiate these equations three times with respect to $u$ (denoted by $'$):
	\[\partial_x(f_i''')=c_{1i}^k f_k''', \qquad \partial_y(f_i''')=\tilde c_{2i}^k(x) f_k'''   \] By the above assumption we have $f_i'''(0,0,u)=0$, and by the uniqueness theorem for systems of linear ODEs it follows that for every $(x,y,u) \in U$ we have $f_i'''(x,y,u)=0$, and therefore
	\begin{equation}
	f_i(x,y,u)=  \alpha_i(x,y)+ \beta_i(x,y) u+ \gamma_i(x,y) u^2. \label{normalform}
	\end{equation}
	Note also that if $f_i''$ (or $f_i'$) vanish on $(0,0,u)$, we may assume $\gamma_i\equiv 0$ (or $\gamma_i\equiv 0$ and $\beta_i \equiv 0$) for every $i$. The last statement  of the theorem follows by the fact that $\dim \varphi_a(\hat{\mathfrak{st}_{a}})$ is the same for every $a \in U$. 
\end{proof}

\begin{definition}
	We say that the lift $\hat{\mathfrak{g}}$ of $\mathfrak g \subset \vf{\mathbb C^2}$ is metric, affine or projective if $\mathfrak h$ is one, two or three dimensional, respectively, and $\varphi_a(\hat{\mathfrak{st}_{a}})$ is of normal form (\ref{classification1D}) at every point $a \in \mathbb C^2$. 
\end{definition}

By theorem \ref{main} all lifts are locally equivalent to one of these three types, so from now on we will restrict to such lifts. This simplifies our computations. Geometrically, we may think about this lifting as choosing a structure on the fiber, namely metric, affine or projective, and requiring the lift to preserve this structure. Another useful observation is that the properties of $\mathfrak{st}_a$ and $\mathfrak h$ are closely linked.

\begin{corollary} \label{quotient}
		 If $\mathfrak{st}_a$ is solvable, then there are no projective lifts.
		 If $\mathfrak{st}_a$ is abelian, then there are no projective or affine lifts. 
\end{corollary}
\begin{proof}
The map $\varphi_a\colon \hat{\mathfrak{st}_a} \to \mathfrak h_a \simeq \mathfrak h$ is a Lie algebra homomorphism, and the image of a solvable (resp. abelian) Lie algebra is solvable (resp. abelian). 
\end{proof}
In particular, from Lie's classification it follows that only the primitive Lie algebras may have projective lifts.

\subsection{Coordinate transformations}
It is natural to consider two lifts to be equivalent if there exist a coordinate transformation on the fibers ($u \mapsto A(x,y,u)$), taking one to the other. 
We consider metric lifts up to translations $u \mapsto u+A(x,y)$, affine lifts up to affine transformations $u \mapsto A(x,y) u+B(x,y)$ and projective lifts up to projective transformations $u \mapsto \frac{A(x,y) u+B(x,y)}{C(x,y) u+D(x,y)}$. The following example shows the general procedure we use for finding lifts.

\begin{example}\label{g6}
	Consider the Lie algebra $\mathfrak{g}_{6}$ which is spanned by vector fields
	\[ X_1=\partial_x,\quad  X_2=\partial_y,\quad X_3= y \partial_y,\quad  X_4= y^2 \partial_y.\] Since the stabilizer of $0$ is solvable, we may by corollary \ref{quotient} assume that the generators of a lift $\hat{\mathfrak g}_{6}$ is of the form $\hat{X}_i = X_i+f_i \partial_u$,  where $f_i$ are affine functions in $u$. All lifts are either metric og affine. 
	
	By example \ref{simplify} we may assume that $f_1 \equiv 0 \equiv f_2$ after making an affine change of coordinates (or a translation if we consider metric lifts). The type of coordinate transformation was not specified in the example, but it is clear that the PDE in the example can be solved within our framework of metric and affine lifts, respectively.
	
	The commutation relations $[X_1,X_3]=0,[X_2,X_3]=X_2$ imply that $f_3$ is a function of $u$ alone. The commutation relations $[X_1,X_4]=0,[X_2,X_4]=2 X_3,[X_3,X_4]=X_4$ result in the differential equations 
	\[(f_4)_x=0, \quad (f_4)_y=2 f_3, \quad y (f_4)_y+f_3 (f_4)_u-f_4 (f_3)_u=f_4.\]
	The first two equations give $f_4=2yf_3(u)+b(u)$.
	After inserting this into the third equation it simplifies to $f_3 b_u-b (f_3)_u=b$.
	
	Since the lift is either metric or affine, we may assume that $f_3=A_0 + A_1 u$ and $b=B_0+B_1 u$. Then the equation above results in $B_1=0$ and $B_0 A_1=-B_0$.
	If $B_0=0$ we get transitive lifts only when $A_1=0$:
	\begin{equation*}
	\partial_x,\quad  \partial_y,\quad y\partial_y+A_0\partial_u,\quad y^2 \partial_y+2A_0 y\partial_u.
	\end{equation*}
	These are metric lifts. If $A_1=-1$ we get the affine lift
	\begin{equation*}
	\partial_x, \quad \partial_y,\quad y \partial_y-u \partial_u,\quad y^2 \partial_y+(1-2yu)\partial_u 
	\end{equation*}
	where $A_0$ and $B_0$ have been normalized by a translation and scaling, respectively. 	
\end{example}

\begin{remark}
	The family of metric lifts is also invariant under transformations of the form $u \mapsto C u+A(x,y)$, where $C$ is constant. However, we would like to restrict to $C=1$. A consequence of this choice is that we get a correspondence between metric lifts and Lie algebra cohomology which will be discussed in section \ref{cohomology}. The same cohomology spaces are treated in \cite{Olver92} where they are used for classifying Lie algebras of differential operators on $\mathbb C^2$. 
	
	We also get a correspondence between metric lifts and ``linear lifts'', whose vector fields act as infinitesimal scaling transformations in fibers. Using the notation above they take the form $\hat X=X+f(x,y) u \partial_u$.
	They make up an important type of lifts, but we do not consider them here due to their intransitivity. Since the transformation $u \mapsto \exp(u)$ takes metric lifts to linear lifts, the theories of these two types of lifts are analogous (given that we allow the right coordinate transformations).  This makes many of the results in this paper applicable to linear lifts as well. 
	As an example the classification of linear lifts under linear transformations ($u \mapsto u A(x,y)$), will be similar to that of metric lifts under translations ($u \mapsto u+A(x,y)$).
\end{remark}

\section{List of lifts} \label{list}
This section contains the list of lifts of the Lie algebras from section \ref{classification2D}  on $\pi\colon \mathbb C^2 \times \mathbb C \to \mathbb C^2$. For a Lie algebra $\mathfrak g \subset \vf{\mathbb C^2}$ we will denote by $\hat{\mathfrak{g}}^m,\hat{\mathfrak{g}}^a,\hat{\mathfrak{g}}^p$ the metric, affine and projective lifts, respectively.

\begin{theorem}
	The following list contains all metric, affine and projective lifts of the Lie algebras from Lie's classification in section \ref{classification2D}. 
\end{theorem}
\begin{align*}
\hat{\mathfrak{g}}_1^m&=\Span{\partial_x, \partial_y, x \partial_y, x \partial_x-y \partial_y, y \partial_x, x \partial_x+y \partial_y+2C \partial_u, \\ 
	&\qquad x^2 \partial_x+xy \partial_y+3 C x \partial_u, xy \partial_x+y^2 \partial_y + 3C y \partial_u}\\
\hat{\mathfrak{g}}_1^p & =\Span{\partial_x, \partial_y, x \partial_y+\partial_u, x \partial_x-y \partial_y-2u \partial_u, y \partial_x - u^2 \partial_u, x \partial_x+y \partial_y, \\
	& \qquad x^2 \partial_x+xy \partial_y+(y-xu) \partial_u, xy \partial_x+y^2 \partial_y+u(y-xu)\partial_u}\\
\hat{\mathfrak{g}}_2^m&= \Span{\partial_x, \partial_y, x \partial_y, x \partial_x-y\partial_y, y \partial_x, x \partial_x+y \partial_y+C\partial_u} \\
\hat{\mathfrak{g}}_2^p&=  \Span{\partial_x, \partial_y, x \partial_y+\partial_u, x \partial_x-y \partial_y-2u\partial_u, y \partial_x-u^2 \partial_u, x \partial_x+y \partial_y} \\
\hat{\mathfrak{g}}_3^p &=  \Span{\partial_x, \partial_y, x \partial_y+\partial_u, x \partial_x-y \partial_y-2u \partial_u, y \partial_x-u^2 \partial_u }\\
\hat{\mathfrak{g}}_4^m&=%\Span{\partial_x, x^i e^{\alpha_j x} \partial_y+b_{j,i}(x) \partial_u}\\ &\text{where } b_{1,0}\equiv 0,\quad  b_{1,i}=e^{\alpha_1 x} \sum_{k=1}^i \binom{i}{k} C_{1,k} x^{i-k} \text{ and } b_{j,i}=e^{\alpha_j x} \sum_{k=0}^i \binom{i}{k} C_{j,k} x^{i-k} \\ 
\Span{\partial_x, x^i e^{\alpha_j x} \partial_y+ e^{\alpha_j x} \left( \sum_{k=0}^i \tbinom{i}{k} C_{j,k} x^{i-k} \right) \partial_u \mid C_{1,0}=0 }  \\
\hat{\mathfrak{g}}_5^m &= \Span{\partial_x,y\partial_y+C\partial_u, x^i e^{\alpha_j x} \partial_y}\\
\hat{\mathfrak{g}}_5^a &= \Span{\partial_x,y\partial_y+u \partial_u, x^i e^{\alpha_j x} \partial_y+ e^{\alpha_j x} \left( \sum_{k=0}^i \tbinom{i}{k} C_{j,k} x^{i-k} \right) \partial_u\mid C_{1,0}=0} \\
\hat{\mathfrak{g}}_6^m&= \Span{\partial_x, \partial_y, y\partial_y +C\partial_u, y^2 \partial_y+2Cy\partial_u}\\
\hat{\mathfrak{g}}_6^a&= \Span{\partial_x, \partial_y, y \partial_y-u \partial_u,y^2 \partial_y+(1-2yu)\partial_u}\\
\hat{\mathfrak{g}}_7^m&= \Span{\partial_x, \partial_y, x\partial_x +C\partial_u, x^2 \partial_x+x \partial_y+2C x\partial_u}\\
\hat{\mathfrak{g}}_7^a&= \Span{\partial_x, \partial_y, x \partial_x-u \partial_u,x^2 \partial_x+x \partial_y+(1-2xu)\partial_u}\\
%\hat{\mathfrak{g}}_8^m &= \Span{\partial_x, \partial_y, x\partial_x+\alpha y \partial_y+ A\partial_u, x^{i} \partial_y + b_i(x) \partial_u},  \\
%& \text{where }b_i(x) \equiv 0 \text { for } i=1,...,s-1,  \text{ and } b_{s+k}(x)= \tbinom{s+k}{s} B x^k \\& \text{for } k=0,...,r-3-s \text{ and } B=0 \text{ unless } \alpha=s\\
%\hat{\mathfrak{g}}_8^a &= \Span{\partial_x, \partial_y, x \partial_x+\alpha y\partial_y+ (\alpha-s)u\partial_u, x^{i} \partial_y + b_i(x) \partial_u},  \\
%& \text{where } \alpha \neq s,  b_i(x) \equiv 0 \text{ for } i=1,...,s-1, \\ &\text{and } b_{s+k}(x) =  \tbinom{s+k}{s} x^k \text{ for } k=0,...,r-3-s\\
\hat{\mathfrak{g}}_8^m &= \Span{\partial_x, \partial_y, x\partial_x+\alpha y \partial_y+ A\partial_u, x \partial_y,...,x^{s-1} \partial_y,\\ & \qquad x^{s+i} \partial_y + \tbinom{s+i}{s} B x^i \partial_u \mid i=0,...,r-3-s},  \\
& \text{where } B=0 \text{ unless } \alpha=s\\
\hat{\mathfrak{g}}_8^a &= \Span{\partial_x, \partial_y, x \partial_x+\alpha y\partial_y+ (\alpha-s)u\partial_u, x \partial_y,...,x^{s-1} \partial_y, \\ & \qquad x^{s+i} \partial_y +\tbinom{s+i}{s} x^i \partial_u \mid i=0,...,r-3-s, \quad \alpha \neq s}\\
\hat{\mathfrak{g}}_9^m &= \Span{\partial_x,  \partial_y, x \partial_x+ ((r-2)y+x^{r-2}) \partial_y+C \partial_u,  x \partial_y, ..., x^{r-3} \partial_y} \\
\hat{\mathfrak{g}}_9^a &= \Span{\partial_x,   \partial_y,  
	x \partial_x+ ((r-2)y+x^{r-2} )\partial_y+ \left(\tbinom{r-2}{s} x^{r-s-2}+(r-s-2) u\right) \partial_u,\\ 
	& \qquad x \partial_y, ..., x^{s-1} \partial_y, x^{s+i} \partial_y + \tbinom{s+i}{s} x^i \partial_u \mid  i=0,...,r-3-s}\\
\hat{\mathfrak{g}}_{10}^m &= \Span{\partial_x,  \partial_y,  x \partial_x+ A \partial_u,  y \partial_y+B \partial_u,  x \partial_y, ..., x^{r-4} \partial_y}\\
\hat{\mathfrak{g}}_{10}^a &= \Span{\partial_x,  \partial_y,  x \partial_x-su \partial_u,  y \partial_y+u\partial_u,  x \partial_y,  ...,  x^{s-1} \partial_y, \\ &\qquad x^{s+i} \partial_y+ \tbinom{s+i}{s} x^i \partial_u \mid  i=0,...,r-4-s}\\
\hat{\mathfrak{g}}_{11}^m &=\Span{\partial_x,  \partial_y,  x \partial_x + A \partial_u,  y \partial_y+B\partial_u,  y^2 \partial_y+2By \partial_u}\\
\hat{\mathfrak{g}}_{11}^a &=\Span{\partial_x,  \partial_y,  x \partial_x ,  y \partial_y-u\partial_u,  y^2 \partial_y+(1-2yu) \partial_u}\\
\hat{\mathfrak{g}}_{12}^m&= \Span{\partial_x,  \partial_y,  x \partial_x+A \partial_u,  y \partial_y+B\partial_u,x^2 \partial_x+2Ax \partial_u , y^2 \partial_y+2By \partial_u}\\
\hat{\mathfrak{g}}_{12}^{a1} &=\Span{\partial_x,  \partial_y,  x \partial_x- u \partial_u ,  y \partial_y,x^2 \partial_x+(1-2xu) \partial_u,  y^2 \partial_y}\\
\end{align*}
\begin{align*}
\hat{\mathfrak{g}}_{12}^{a2} &=\Span{\partial_x,  \partial_y,  x \partial_x ,  y \partial_y- u \partial_u,x^2 \partial_x,  y^2 \partial_y+(1-2yu) \partial_u}\\
\hat{\mathfrak{g}}_{13}^{m1} &= \Span{\partial_x, \partial_y,  x \partial_x+ y \partial_y+A \partial_u, x \partial_y+B \partial_u, x^2 \partial_y+2B x \partial_u,   \\ & \qquad x^2 \partial_x+2 xy \partial_y+(2x A+2yB)\partial_u }\\
\hat{\mathfrak{g}}_{13}^{m2} &= \Span{\partial_x, \partial_y, x \partial_x+\tfrac{r-4}{2} y \partial_y+C\partial_u, x \partial_y,...,x^{r-4} \partial_y, \\ & \qquad x^2 \partial_x+(r-4) xy \partial_y+2Cx \partial_u}\\
\hat{\mathfrak{g}}_{13}^{a1} &= \Span{\partial_x, \partial_y, x \partial_x+\tfrac{r-4}{2} y \partial_y-u\partial_u,  x \partial_y,...,x^{r-4} \partial_y,\\ & \qquad x^2 \partial_x+(r-4) xy \partial_y+(1-2xu) \partial_u}\\
\hat{\mathfrak{g}}_{13}^{a2} &= \Span{\partial_x,  \partial_y,  x^2 \partial_x+(r-4)xy \partial_y+(x (r-6) u+(r-4)y) \partial_u, \\  &\qquad x \partial_x+\tfrac{r-4}{2} y \partial_y+\tfrac{r-6}{2} u \partial_u,  x^i \partial_y+i x^{i-1} \partial_u \mid i=1,...,r-4 }\\
\hat{\mathfrak{g}}_{14}^{m} &= \Span{\partial_x,  \partial_y,  x \partial_x+A\partial_u,  y \partial_y+B \partial_u, x \partial_y, ..., x^{r-5} \partial_y,\\
	& \qquad x^2 \partial_x+(r-5) x y \partial_y+(2A+(r-5) B) x \partial_u} \\
%\hat{\mathfrak{g}}_{14}^{a1} &= \Span{\partial_x,  \partial_y,  x \partial_x- u\partial_u,  y \partial_y+u \partial_u, 
%	x^2 \partial_x+2 x y \partial_y+2 y \partial_u,\\  & \qquad x \partial_y+ A \partial_u, x^2 \partial_y+ (2A x+B) \partial_u}\\
\hat{\mathfrak{g}}_{14}^{a1} &= \Span{\partial_x,  \partial_y,  x \partial_x- u\partial_u,  y \partial_y,
	x \partial_y,...,x^{r-5} \partial_y, \\ & \qquad x^2 \partial_x+(r-5) x y \partial_y+(1-2 xu) \partial_u}\\
\hat{\mathfrak{g}}_{14}^{a2} &= \Span{	\partial_x,  \partial_y,  x^2 \partial_x+(r-5) x y \partial_y+((r-7) x u+(r-5) y) \partial_u, \\
	& \qquad x \partial_x- u\partial_u,  y \partial_y+u \partial_u,  x^i \partial_y+ i x^{i-1}  \partial_u \mid i=1,...,r-5}\\
\hat{\mathfrak{g}}_{15}^m &= \Span{\partial_x, x \partial_x + \partial_y,  x^2 \partial_x+2x \partial_y+C e^y\partial_u }\\
\hat{\mathfrak{g}}_{16}^m &= \Span{\partial_x,  x \partial_x-y\partial_y+C \partial_u,  x^2 \partial_x+(1-2xy) \partial_y+ 2C x\partial_u}
\end{align*}

The proof of theorem \ref{list} is a direct computation following the algorithm described above. The computations are not reproduced here, beyond example \ref{g6}, but they can be found in the ancillary file to the arXiv version of this paper. 

All capital letters in the list denote complex constants. For the metric lifts, one of the constants can always be set equal to $1$ if we allow to rescale $u$. In the affine lift $\hat{\mathfrak g}_5^a$ one of the constants must be nonzero in order for the lift to be transitive, and it can be set equal to $1$ by a scaling transformation. Notice also that even though ${\mathfrak{g}}_{15}$ is not locally equivalent to ${\mathfrak{g}}_{16}$, their lifts are locally equivalent. In addition the two affine lifts of $\mathfrak g_{12}$ are locally equivalent. 

Most of this list already exist in the literature. The lifts of the three primitive Lie algebras can be found in \cite{Transformationsgruppen3}. The first attempt to give a complete list of imprimitive Lie algebras of vector fields on $\mathbb C^3$ was done by Amaldi in \cite{Amaldi1,Amaldi2}. Most of the Lie algebras we have found is contained in Amaldi's list of Lie algebras of ``type A'', but some of our lifts are missing. Examples are $\hat{\mathfrak{g}}_{10}^a,\hat{\mathfrak{g}}_{14}^{m},\hat{\mathfrak{g}}_{14}^{a1}$ and $\hat{\mathfrak{g}}_{8}^a$ with general $\alpha$ and $B=0$. There is in error in the Lie algebra corresponding to $\hat{\mathfrak{g}}_{14}^{a1}$ which was also noticed in \cite{Hillgarter1, Hillgarter2}. The lifts of nonsolvable Lie algebras are contained in \cite{Doubrov}, and the case of metric lifts was also considered in \cite{MasterThesis}.

\section{Metric lifts and Lie algebra cohomology} \label{cohomology}
We conclude this treatment by showing that there is a one-to-one correspondence between the metric lifts of $\mathfrak g \subset \vf{\mathbb C^2}$ and the Lie algebra cohomology group $H^1(\mathfrak g, C^\omega (\mathbb C^2))$. The main result is analogous to \cite[Theorem 2]{Olver92}. 

A metric lift of a Lie algebra $\mathfrak g \subset \vf{\mathbb C^2}$ is given by a $C^\omega(\mathbb C^2)$-valued one-form $\psi$ on $\mathfrak g$. For vector fields  $X,Y \in \mathfrak g$ lifted to  $\hat X = X + \psi_X \partial_u$ and $\hat Y = X + \psi_Y \partial_u$  we have
\begin{equation}
[\hat X,\hat Y] = [X+\psi_X \partial_u, Y+\psi_Y \partial_u] = [X,Y] + (X(\psi_Y)-Y(\psi_X))\partial_u. \label{xhatyhat}
\end{equation}
Consider the first terms of the Chevalley-Eilenberg complex
\begin{equation*}
0 \longrightarrow C^\omega(\mathbb C^2) \overset{d}{\longrightarrow} \mathfrak g^* \otimes C^\omega(\mathbb C^2) \overset{d}{\longrightarrow} \Lambda^2 \mathfrak g^* \otimes C^\omega (\mathbb C^2)
\end{equation*}
where the differential $d$ is defined by 
\begin{align*}
df(X) &= X(f), \quad f \in C^\omega(\mathbb C^2) \\
d\psi(X,Y) &= X(\psi_Y)-Y(\psi_X) - \psi_{[X,Y]}, \quad \psi \in \mathfrak g^* \otimes C^\omega(\mathbb C^2).
\end{align*}
This complex depends not only on the abstract Lie algebra, but also on its realization as a Lie algebra of vector fields.  It is clear from (\ref{xhatyhat}) that $\psi \in \mathfrak g^* \otimes C^\omega(\mathbb C^2)$ corresponds to a metric lift if and only if $d \psi =0$. 

Two metric lifts are equivalent if there exists a biholomorphism \[\phi\colon (x,y,u) \mapsto (x,y,u-U(x,y))\] on $\mathbb C^2 \times \mathbb C$ that brings one to the other. A lift of $X$  transforms according to
\begin{equation*}
d \phi \colon X + \psi_X \partial_u \mapsto X + (\psi_X - dU(X))\partial_u
\end{equation*}
which shows that two lifts are equivalent if the difference between their defining one-forms is given by $dU$ for some  $U\in C^\omega(\mathbb C^2)$. Thus we have the following theorem, relating the cohomology space
\begin{equation*}
H^1(\mathfrak g, C^\omega(\mathbb C^2)) = \{\psi \in\mathfrak g^* \otimes C^\omega(\mathbb C^2) \mid d \psi =0\}/\{dU \mid U \in C^\omega(\mathbb C^2)\},
\end{equation*}
to the space of metric lifts.
\begin{theorem} \label{cohomologytheorem}
	There is a one-to-one correspondence between the space of metric lifts of the Lie algebra $\mathfrak g \subset \vf{\mathbb C^2}$ (up to equivalence) and the cohomology space $H^1(\mathfrak g, C^\omega (\mathbb C^2))$. 
\end{theorem}

The theorem gives a transparent interpretation of metric lifts, while also showing a way to compute $H^1(\mathfrak g, C^\omega (\mathbb C^2))$, through example \ref{g6}. This method is essentially the one that was used in \cite{Olver92}, where the same cohomologies were found. There they extended Lie's classification of Lie algebras of vector fields to Lie algebras of first order differential operators on $\mathbb C^2$, and part of this work is equivalent to our classification of metric lifts. 

Their results coincide with ours, with the  exceptions  $\mathfrak g_8$ which corresponds to case 5 and 20 in \cite{Olver92} and $\mathfrak g_{16}, \mathfrak g_{15}, \mathfrak g_7$ which correspond to cases 12, 13 and 14, respectively. For $\mathfrak g_8$ it seems like they have not considered the case corresponding to $\ker(d \pi|_{\hat{\mathfrak g}})={0}$ which is the only case we consider. The realizations used in \cite{Olver92} for cases 12, 13 and 14 have singular orbits, while their cohomologies are computed after restricting to subdomains, avoiding singular orbits. The cohomology is sensitive to choice of realization as Lie algebra of vector fields, and will in general change by restricting to a subdomain. The following example, based on realizations of $sl(2)$, illustrates this. 
\begin{example}
	The metric lift 
	\[\hat{\mathfrak{g}}_{16}^m = \Span{\partial_x,  x \partial_x-y\partial_y+C \partial_u,  x^2 \partial_x+(1-2xy) \partial_y+ 2C x\partial_y} \] is parametrized by a single constant, and thus $H^1(\mathfrak g_{16}, C^\omega (\mathbb C^2)) = \mathbb C$. Similarly, we see that $H^1(\mathfrak g_{15}, C^\omega (\mathbb C^2)) = \mathbb C$. 
	
	The  Lie algebra
	$  \tilde{\mathfrak{g}}_{16} = \Span{\partial_x, x \partial_x+ y \partial_y, x^2 \partial_x+y(2x+y)  \partial_y  }$ is related to \cite[case 12]{Olver92} by the transformation $y \mapsto x+y$. It is also
	is locally equivalent to $\mathfrak g_{16}$, but it has a singular one-dimensional orbit, $y=0$. Its metric lift is given by 
	\[ \Span{\partial_x,  x\partial_x+ y \partial_y+A \partial_u, x^2 \partial_x+y(2x+y)  \partial_y  +(2A x+B y) \partial_u}\]
	which implies that $H^1(\tilde{\mathfrak{g}}_{16}, C^\omega (\mathbb C^2)) = \mathbb C^2$.
	
	The Lie algebra $ \tilde{\mathfrak{g}}_{15}=\Span{y \partial_x, x \partial_y,x \partial_x-y\partial_y}$ is the standard representation on $\mathbb C^2$. If we split $C^\omega ( \mathbb C^2)= \oplus_{k=0}^\infty S^k (\mathbb C^2)^*$, then $H^1(\tilde{\mathfrak{g}}_{15},C^\omega(\mathbb C^2))= \oplus_{k=0}^\infty H^1(\tilde{\mathfrak{g}}_{15},S^k (\mathbb C^2)^*)$. By Whitehead's lemma, since $S^k(\mathbb C^2)^*$ is a finite-dimensional module over $\tilde{\mathfrak{g}}_{15}$, the cohomologies $H^1(\mathfrak g,S^k(\mathbb C^2)^*)$ vanish, and thus $H^1(\tilde{\mathfrak{g}}_{15},C^\omega(\mathbb C^2))=0$. Hence the cohomologies of the locally equivalent Lie algebras $\mathfrak{g}_{15}$ and $\tilde{\mathfrak{g}}_{15}$ are different. To summarize, we have two pairs of locally equivalent realizations of $sl(2)$, and their cohomologies are
	\begin{align*}
	H^1(\mathfrak g_{16}, C^\omega (\mathbb C^2)) = \mathbb C, \qquad &H^1(\tilde{\mathfrak{g}}_{16}, C^\omega (\mathbb C^2)) = \mathbb C^2, \\ H^1(\mathfrak g_{15}, C^\omega (\mathbb C^2)) = \mathbb C, \qquad &H^1(\tilde{\mathfrak{g}}_{15},C^\omega(\mathbb C^2))=0.
	\end{align*}
	%They are related by the coordinate change $x'=x/y, y'=-2 \log(y)$. 
\end{example}

	The Lie algebra cohomologies considered in this paper are related to the relative invariants (and singular orbits) of the corresponding Lie algebras of vector fields \cite{FelsOlver}.  A consequence of \cite[Theorem~5.4]{FelsOlver} is that a locally transitive Lie algebra $\mathfrak g$ of vector fields has a scalar relative invariant if it has a nontrivial metric lift whose orbit-dimension is equal to that of $\mathfrak g$. The Lie algebra $\tilde{\mathfrak g}_{16}$ has two-dimensional orbits when $A=B$. Therefore there exists an absolute invariant, and it is given by $e^u/y^A$. The corresponding relative invariant of $\mathfrak g_{16}$ is $y^A$ and it defines the singular orbit $y=0$. 

\section*{Acknowledgements}
I would like to thank Boris Kruglikov for his invaluable guidance throughout this work.

\end{document}